\documentclass[12pt]{article}
 
\usepackage{amsmath,amssymb,amstext,mathtools,amsthm,mathdots}
\usepackage[affil-it]{authblk}
\usepackage[pdftex]{graphicx} 
\usepackage{mathtools} 
\usepackage{titleps}
\usepackage{relsize}

\newcommand{\TITLE}{Fixed points of Lie group actions on moduli spaces}
\newcommand{\SUBTITLE}{A tale of two actions}
\newcommand{\AUTHOR}{C. J. Lang}

\usepackage[pdftex,pagebackref=true]{hyperref}

\hypersetup{
    plainpages=false,       
    unicode=false,         
    pdftoolbar=true,      
    pdfmenubar=true,     
    pdffitwindow=false,   
    pdfstartview={FitH},    
    pdftitle={\TITLE: \SUBTITLE},    
    pdfauthor={\AUTHOR},    
    pdfsubject={Mathematics}, 
    pdfkeywords={Lie theory} {moduli space} {representation theory}, 
    pdfnewwindow=true,      
    colorlinks=true,       
    linkcolor=blue,        
    citecolor=green,        
    filecolor=magenta,      
    urlcolor=cyan           
}

\usepackage{bookmark}

\providecommand{\keywords}[1]
{
  {\small	
  \textbf{\textit{Keywords---}} #1}
}

\usepackage{cite}
\makeatletter
\newcommand{\citecomment}[2][]{\citen{#2}#1\citevar}
\newcommand{\citeone}[1]{\citecomment{#1}}
\newcommand{\citetwo}[2][]{\citecomment[,~#1]{#2}}
\newcommand{\citevar}{\@ifnextchar\bgroup{;~\citeone}{\@ifnextchar[{;~\citetwo}{]}}}
\newcommand{\citefirst}{\@ifnextchar\bgroup{\citeone}{\@ifnextchar[{\citetwo}{]}}}

\makeatother

\renewpagestyle{plain}{
\setfoot[][][]{}{\thepage}{}
}

\newtheorem{theorem}{Theorem}
\numberwithin{theorem}{section}

\newtheorem{note}[theorem]{Note}
\newtheorem{prop}[theorem]{Proposition}

\newtheorem{example}[theorem]{Example}

\usepackage[margin=1in]{geometry}

\begin{document}

\pagestyle{plain}
\pagenumbering{arabic}

\title{\TITLE\\[0.2em]\smaller{}\SUBTITLE}
\author{\AUTHOR\thanks{E-mail address: cjlang@uwaterloo.ca}}
\affil{Department of Pure Mathematics \\ 200 University Avenue West \\ University of Waterloo, Canada \\ N2L 3G1}
\date{January 14, 2025} %***TODO: Change for Arxiv to current date***
\maketitle
\begin{abstract}
In this paper, we examine Lie group actions on moduli spaces (sets themselves built as quotients by group actions) and their fixed points. We show that when the Lie group is compact and connected, we obtain a linear constraint. This constraint makes the problem of finding fixed points one of representation theory, greatly simplifying the search for such points. We obtain a similar result when the Lie group is one-dimensional. For compact and disconnected Lie groups, we show that we need only additionally check a finite number of points. Finally, we show that the subgroup fixing an equivalence class in the moduli space is a compact Lie subgroup.
\end{abstract}

\section{Introduction}\label{sec:intro}
In this paper, we examine Lie group actions on moduli spaces (sets themselves built as quotients by group actions) and their fixed points. While there are countless examples of people examining fixed points of group actions, the literature lacks a systematic study of actions on moduli spaces, where identifying orbits is more challenging. In this paper, we begin to rectify this by studying the fixed points of Lie group actions on moduli spaces.

In particular, we prove two main results: Theorem~\ref{thm:mainthm} and Proposition~\ref{prop:R}, the former appearing in my thesis~\cite[Theorem~1.1.1]{lang_thesis_2024}. These results tell us that these fixed points must satisfy a linear constraint and the problem of finding such points is one of representation theory.

Before stating these results precisely, we establish our notation. Let $M$ be a smooth manifold and $G$ a Lie group acting smoothly on $M$. From this Lie group action, we obtain a Lie algebra action. If we are dealing with right (left) group actions, we obtain Lie algebra (anti-)homomorphisms taking value in $\mathfrak{X}(M)$. Given $y\in\mathrm{Lie}(G)$ and $A\in M$, the Lie algebra action is given by $y.A:=\frac{d}{dt}\bigr|_{t=0}\mathrm{exp}(ty)\cdot A\in T_AM$. 

Given this notation, we can introduce our results. Theorem~\ref{thm:mainthm} deals with compact Lie groups.
\begin{theorem}
Let $\mathcal{X}$ be a smooth manifold, $\mathcal{G}$ a compact Lie group, and $\mathcal{S}$ a compact, connected Lie group. Suppose that $\mathcal{G}$ and $\mathcal{S}$ act smoothly on $\mathcal{X}$ on the left and the two actions commute. We have that $[A]\in\mathcal{X}/\mathcal{G}$ is fixed by $\mathcal{S}$ if and only if there is some Lie algebra homomorphism $\rho\colon\mathrm{Lie}(\mathcal{S})\rightarrow \mathrm{Lie}(\mathcal{G})$ such that, for all $x\in\mathrm{Lie}(\mathcal{S})$, \label{thm:mainthm}
\begin{equation}
x.A+\rho(x).A=0.\label{eq:maineq}
\end{equation}
\end{theorem}
When $\mathcal{S}$ is one-dimensional, using a similar method, we obtain the same result even when we no longer have compactness. %As an example of the use of this result, I use it to study instantons with circular symmetry in my thesis~\cite[Theorem~3.2.1]{lang_thesis_2024}.
\begin{prop}
Let $\mathcal{X}$ be a smooth manifold, $\mathcal{G}$ a Lie group, and $\mathcal{S}$ a connected, one-dimensional Lie group (isomorphic to either $S^1$ or $\mathbb{R}$). Suppose that $\mathcal{G}$ and $\mathcal{S}$ act smoothly on $\mathcal{X}$ on the left and the two actions commute. We have that $[A]\in\mathcal{X}/\mathcal{G}$ is fixed by $\mathcal{S}$ if and only if there is some $\rho\in\mathrm{Lie}(\mathcal{G})$ such that, for all $t\in\mathbb{R}$, \label{prop:R}
\begin{equation}
t.A+t\rho.A=0.\label{eq:R}
\end{equation}
\end{prop}

Theorem~\ref{thm:mainthm} and Proposition~\ref{prop:R} are the result of previous work finding novel symmetric topological solitons in gauge theories. However, they are general results that can also be used outside of gauge theory to easily find fixed points of Lie group actions on moduli spaces. The rest of this section is devoted to motivating these results and ends with an outline of this paper. Throughout the following, we borrow motivation from my thesis~\cite[\S1.1]{lang_thesis_2024}.

Broadly speaking, gauge theory involves two ingredients: objects (solutions to some equation) and a gauge action (a group action preserving solutions to the equation). The underlying theme is that we only care about the object up to the gauge action. Denoting the set of objects by $\mathcal{X}$ and the gauge group $\mathcal{G}$, in gauge theory we only care about the moduli space $\mathcal{X}/\mathcal{G}$. 

The classical example of a gauge theory is electrodynamics. Here, the objects are one-forms representing electromagnetic potentials and the gauge group is the group of smooth maps $g\colon\mathbb{R}^4\rightarrow \mathrm{U}(1)$ acting on the one-forms via $g\cdot A:=gAg^{-1}-(dg)g^{-1}$. In this case, we are interested in the curvature form $F_A=dA+A\wedge A$, as it encodes the electric and magnetic fields of the system. We see that $F_{g\cdot A}=F_A$, so these fields and the underlying Lagrangian are unaffected by the action of the gauge group (gauge transformations). As we can only measure the electric and magnetic fields, not the one-forms, we only care about the one-forms up to the gauge action.

Suppose that we are trying to find examples of some gauge theoretic object $\mathcal{X}$ with gauge group $\mathcal{G}$. We wish to impose a group of symmetries $\mathcal{S}$ to help simplify our search. Naively, we may enforce symmetry on $\mathcal{X}$ directly, by looking for points on $\mathcal{X}$ fixed by $\mathcal{S}$. However, since we only care about elements in the moduli space, we can have a more general notion of symmetry. Indeed, since elements in the same equivalence class are physically identical, we want them to share the same symmetries. Hence, we need the symmetry action to descend to an action on the moduli space, which happens when the gauge and symmetry actions commute. 

We then search for fixed points of the symmetry action on the moduli space. An object $A$ in $\mathcal{X}$ whose equivalence class $[A]$ in $\mathcal{X}/\mathcal{G}$ is fixed by $\mathcal{S}$ is considered symmetric, as acting by $\mathcal{S}$ produces a new object that only differs from $A$ by a gauge transformation. That is, as far as we are concerned, the object is unaffected by $\mathcal{S}$.

While fixed points on the moduli space are more plentiful than fixed points in $\mathcal{X}$, finding them is much harder. Indeed, we are not just looking at objects $A\in\mathcal{X}$ such that for all $s\in\mathcal{S}$, $s\cdot A=A$. Instead, we are looking at objects $A\in\mathcal{X}$ such that for all $s\in\mathcal{S}$, there exists some $g\in\mathcal{G}$ such that $s\cdot A=g\cdot A$. That is, objects where the symmetry action is undone by the gauge action, introducing many new variables to the problem.

Although the literature lacks a systematic study of actions on moduli spaces, there are plenty of results in specific examples. Indeed, a great deal of moduli spaces have been considered using a variety of group actions. I was introduced to the concept through topological solitons. Topological solitons are gauge-theoretic objects that often satisfy complicated non-linear constraints, leaving us with few non-trivial examples of these objects. Symmetries simplify these constraints and make it much easier to find explicit examples, which we can study to uncover properties of these objects in general. While my previous work has focused on finding topological solitons with continuous symmetries, much work has gone into finding examples of such objects---for example instantons, Euclidean monopoles, hyperbolic monopoles, and calorons---with finite symmetries~\cite{leese_stable_1994, hitchin_symmetric_1995, houghton_octahedral_1996, houghton_tetrahedral_1996, houghton_sun_1997, houghton_instanton_1999, singer_symmetric_1999,  sutcliffe_instantons_2004, sutcliffe_platonic_2005, braden_tetrahedrally_2010, allen_adhm_2013, cockburn_symmetric_2014, manton_platonic_2014, cork_symmetric_2018, whitehead_integrality_2022}. However, the methods used to find these examples are specific to the corresponding objects and are not widely applicable beyond these cases.

The aforementioned examples were found using the relationships between the various topological solitons and related objects. For instance, the examples of instantons and hyperbolic monopoles were found using their corresponding ADHM data---quaternionic matrices satisfying non-linear constraints. The topological solitons themselves do not fit into the framework studied in this paper. Indeed, the space of objects and the space of gauge transformations are infinite-dimensional. However, the related objects do fit in this framework and previous work of mine has used these related objects to identify novel examples of instantons, hyperbolic monopoles, and Euclidean monopoles~\cite{charbonneau_construction_2022, lang_hyperbolic_2023, lang_thesis_2024, lang_instantons_2024}.

On a similar note, work has been done looking at instantons with varying circular symmetries~\cite{braam_boundary_1990, manton_platonic_2014, beckett_equivariant_2020}. This work also utilized the relationships between instantons and ADHM data. Unlike the finite symmetries examined above, these circular symmetries were not studied in order to find examples of these objects. Instead, they were studied because instantons with these symmetries correspond to other topological solitons: singular and hyperbolic monopoles. Just as before, the methods used to study instantons with circular symmetries are again very specific to instantons and not applicable broadly.

Similar to the above case of instantons with circular symmetry, many have examined points on moduli spaces fixed by a group action for reasons other than finding examples of objects. For example, given finite group actions, the fixed point locus contained in the moduli spaces of certain stable vector bundles and Higgs bundles have been studied~\cite{narasimhan_generalised_1975, biswas_anti-holomorphic_2015, schaffhauser_finite_2016, biswas_involutions_2019, garcia-prada_involutions_2019, garcia-prada_finite_2020}. In particular, in many cases, the fixed point locus corresponds to the union of moduli spaces of objects equipped with extra structure~\cite{schaffhauser_finite_2016, garcia-prada_finite_2020}.

Above, the fixed point locus in the moduli space is decomposed using group cohomology. This decomposition requires additional structure that is not present in this work. As such, we do not examine the fixed point locus in detail. Instead, we focus on identifying fixed points given a Lie group action and the group of symmetries given a fixed point.

Similar to the above examples, the moduli space of representations of the fundamental group of a compact surface in a connected, semi-simple real Lie group has been studied. Points in this moduli space fixed by finite subgroups of the mapping class group correspond to Higgs bundles equipped with a twisted equivariant structure~\cite{garcia-prada_action_2020}. Another example involving the mapping class group involves the moduli space of irreducible flat $\mathrm{SU}(n)$-connections on an oriented compact surface. For infinitely many values of $n$, there are points on the moduli space fixed by various subgroups of the mapping class group~\cite{andersen_fixed_1997}.

Hilbert schemes are themselves moduli spaces and are closely related to the moduli spaces of instantons~\cite{nakajima_lectures_1999}. The fixed point set of an anti-symplectic involution on the Hilbert scheme of points on a smooth, complex quasi-projective surface is a complex Lagrangian subvariety. The mixed Hodge structures of the cohomolgy groups of this subvariety have been computed and its connected components and their mixed Hodge structures have been classified~\cite{baird_cohomology_2024}. 

Like Hilbert schemes, quiver varieties themselves are moduli spaces. In particular, Nakajima quiver varieties were originally motivated by studying instantons on ALE spaces. The fixed points of the action of abelian reductive subgroups on Nakajima quiver varieties were used to study finite dimensional representations of the quantum affine algebra~\cite{nakajima_quiver_2001}. More recently, moduli spaces of quiver representations have been studied. In particular, the actions of finite groups of quiver automorphisms on these moduli spaces have been examined and the fixed point locus has been decomposed using group cohomology~\cite{hoskins_group_2019}. Additionally, actions of the absolute Galois group of a perfect field on this moduli space have also been considered. Similar to previous cases, the fixed point locus can be decomposed into the union of moduli spaces of twisted quiver representation~\cite{hoskins_rational_2020}.

This work differs from the above in that we examine general Lie group actions instead of finite, abelian, or absolute Galois group actions. This choice of action allows us to study different kinds of symmetries. Moreover, by choosing Lie group actions, we can make use of smoothness to differentiate the equations of symmetry, providing a linear constraint, which greatly simplifies any potential non-linear constraints. Indeed, my previous work using these results have identified many novel topological solitons~\cite{charbonneau_construction_2022, lang_hyperbolic_2023, lang_thesis_2024, lang_instantons_2024}. Additionally, this work differs as we do not restrict ourselves to a specific moduli space and group action.

In Section~\ref{sec:fixed}, we prove Theorem~\ref{thm:mainthm}, which deals with compact Lie groups %, which tells us that if we satisfy the proper technical conditions, then we can differentiate the equations of symmetry. 
and first appeared in my thesis~\cite[Theorem~1.1.1 \& Appendix~A]{lang_thesis_2024}. %Thus, we can use linear algebra to find every example of our symmetric objects. 
As compact Lie groups can be realized as matrix Lie groups, the Lie algebra homomorphisms give us representations of Lie algebras, so we reduce the problem to the realm of representation theory. We also prove Proposition~\ref{prop:R}, which allows us to differentiate the equations of symmetry when the symmetry group is one-dimensional but not necessarily compact. In Section~\ref{subsec:example}, we use Proposition~\ref{prop:R} to find symmetric points in an example. For in depth examples of the use of this result, see my thesis as well as my previous papers, where I identify novel examples of symmetric instantons as well as hyperbolic and Euclidean monopoles~\cite{charbonneau_construction_2022, lang_hyperbolic_2023, lang_thesis_2024, lang_instantons_2024}. In Section~\ref{subsec:closed}, we show that the group of symmetries of an object is a closed Lie subgroup, which helps narrow down the kinds of symmetries that we need to study when looking at objects with differing continuous symmetries.

\section{Fixed points}\label{sec:fixed}
We are primarily interested in fixed points when the Lie group acting on the moduli space is compact; however, we also look at a case when the Lie group is not compact.

It turns out that we need only thoroughly check that a point in the moduli space is fixed by one connected component of the Lie group. Otherwise, we need only check that the point in the moduli space is fixed by a single point in each conected component. In what follows, unless indicated otherwise, the side on which a Lie group acts does not affect the result.
\begin{prop}
Let $\mathcal{X}$ be a smooth manifold and let $\mathcal{G}$ as well as $\mathcal{S}$ be Lie groups. Denote the connected component of $\mathcal{S}$ containing the identity by $\mathcal{S}_0\subseteq\mathcal{S}$. Suppose that $\mathcal{G}$ and $\mathcal{S}$ act smoothly on $\mathcal{X}$ and the two actions commute. We have that $[A]\in\mathcal{X}/\mathcal{G}$ is fixed by $\mathcal{S}$ if and only if $[A]$ is fixed by some element in each connected component of $\mathcal{S}$ and $[A]$ is fixed by $\mathcal{S}_0$. \label{prop:mainprop}
\end{prop}

\begin{proof}
If $[A]$ is fixed by $\mathcal{S}$, then it is fixed by every element in every connected component of $\mathcal{S}$. Hence, we focus on the other direction. Suppose that $[A]$ is fixed by $\mathcal{S}_0$ and some element in each connected component of $\mathcal{S}$. Let $\pi\colon\mathcal{S}\rightarrow\mathcal{S}/\mathcal{S}_0$ be the quotient map sending an element of the group to its connected component and let $\iota\colon\mathcal{S}_0\hookrightarrow\mathcal{S}$ be the inclusion map. Consider the following short exact sequence of groups: 
\begin{equation*}
0\rightarrow \mathcal{S}_0\xhookrightarrow{\iota}\mathcal{S}\xrightarrow{\pi} \mathcal{S}/\mathcal{S}_0\rightarrow 0.
\end{equation*}

Let $x\in \mathcal{S}$. Specifically, let $x$ be in some connected component $\mathcal{S}_i\subseteq\mathcal{S}$. Let $x_i\in\mathcal{S}_i$ fix $[A]$. As $x$ and $x_i$ are in the same component of $\mathcal{S}$, $\pi(x)=\pi(x_i)$, so $x_i^{-1}x\in\mathrm{ker}(\pi)$. As the above sequence is exact, there is some $y\in\mathcal{S}_0$ such that $x_i^{-1}x=y$. Rearranging, $x=x_iy$. Then we see that as $x_i$ and $y$ fix $[A]$ and we have a group action, $[A]$ is fixed by $x$. As $x$ is arbitrary, $[A]$ is fixed by the whole of $\mathcal{S}$.
\end{proof}

If the symmetry group $\mathcal{S}$ is compact, then it has a finite number of connected components. Indeed, we know that $\pi_0(\mathcal{S})=\mathcal{S}/\mathcal{S}_0$. As $\mathcal{S}/\mathcal{S}_0$ is the continuous image of a compact set, it is compact itself. However, $\pi_0(\mathcal{S})$ consists only of isolated points, so $\mathcal{S}/\mathcal{S}_0$ is finite.

As such, if $\mathcal{S}$ is compact, then we need only check that a point is fixed by $\mathcal{S}_0$ and a finite number of additional points. In this case, $\mathcal{S}_0$ is a compact and connected Lie group, falling under the umbrella of Theorem~\ref{thm:mainthm}.
%
%Let $M$ be a smooth manifold and $G$ a Lie group acting smoothly on $M$. Due to the smoothness of Lie group actions, we obtain Lie algebra actions. If we are dealing with right (left) group actions, we obtain Lie algebra (anti-)homomorphisms taking value in $\mathfrak{X}(M)$. Given $y\in\mathrm{Lie}(G)$ and $A\in M$, the Lie algebra action is given by $y.A:=\frac{d}{dt}\bigr|_{t=0}\mathrm{exp}(ty).A\in T_AM$. 
%
%We are now ready to prove the following result.
%\begin{theorem}
%Let $\mathcal{X}$ be a smooth manifold, $\mathcal{G}$ a compact Lie group, and $\mathcal{S}$ a compact, connected Lie group. Suppose that $\mathcal{G}$ and $\mathcal{S}$ act smoothly on $\mathcal{X}$ on the left and the two actions commute. We have that $[A]\in\mathcal{X}/\mathcal{G}$ is fixed by $\mathcal{S}$ if and only if there is some Lie algebra homomorphism $\rho\colon\mathrm{Lie}(\mathcal{S})\rightarrow \mathrm{Lie}(\mathcal{G})$ such that, for all $x\in\mathrm{Lie}(\mathcal{S})$, \label{thm:mainthm}
%\begin{equation}
%x.A+\rho(x).A=0.\label{eq:maineq}
%\end{equation}
%\end{theorem}
%
%\begin{note}
%Nothing changes in \eqref{eq:maineq} if we have two right group actions. However, if we have a mix of left and right group actions, then \eqref{eq:maineq} becomes $x.A-\rho(x).A=0$. This is because, in order to maintain its group structure, the definition of $S_A$ must change depending on which side the Lie groups act on. The same is true for \eqref{eq:R}.\label{note:sign}
%\end{note}
%
\begin{proof}[Proof of Theorem~\ref{thm:mainthm}]
Suppose that $[A]\in\mathcal{X}/\mathcal{G}$ is fixed by $\mathcal{S}$. Let $S_A\subseteq \mathcal{S}\times \mathcal{G}$ be the stabilizer group of $A$, given by
\begin{equation}
S_A:=\{(s,g)\mid s\cdot g\cdot A=A\}.\label{eq:symmetryS}
\end{equation}
Note that the identity belongs to $S_A$ and the commutativity of the actions imply that $S_A$ is a subgroup. 

As $\mathcal{S}$ is connected, it has a simply-connected universal cover $\tilde{\mathcal{S}}$. Denote the covering map $\pi\colon\tilde{\mathcal{S}}\rightarrow\mathcal{S}$. As $[A]$ is fixed by $\mathcal{S}$, the projection map $\pi_1|_{S_A}\colon S_A\rightarrow \mathcal{S}$ is surjective. Our goal is to find a smooth map $\tilde{\mathcal{S}}\rightarrow S_A$ that, when composed with $\pi_1|_{S_A}$, gives us the covering map $\pi$. This map allows us to differentiate the equation of symmetry defining $S_A$.

We start by showing that $S_A$ is a Lie group. Let $f\colon \mathcal{S}\times\mathcal{G}\rightarrow \mathcal{X}$ be the smooth map 
\begin{equation*}
f(s,g):=s\cdot g\cdot A.
\end{equation*}
Then $S_A=f^{-1}(A)$, meaning $S_A$ is closed. By the Closed-subgroup Theorem, as $S_A$ is a closed subgroup of a Lie group, $S_A$ is a Lie group~\cite[Theorem~20.12]{lee_smooth_2012}. Moreover, as $\mathcal{S}\times\mathcal{G}$ is compact and $S_A$ is closed, $S_A$ is compact.

Consider the Lie algebra homomorphism $\phi:=d_{(e_\mathcal{S},e_\mathcal{G})}(\pi_1|_{S_A})\colon \mathrm{Lie}(S_A)\rightarrow \mathrm{Lie}(\mathcal{S})$. Note that $\mathrm{ker}(\phi)\subseteq\mathrm{Lie}(S_A)$ is an ideal of $\mathrm{Lie}(S_A)$. As $S_A$ is compact, it has a bi-invariant metric corresponding to a $\mathrm{Ad}(S_A)$-invariant inner product on $\mathrm{Lie}(S_A)$. Notably, this inner product satisfies
\begin{equation*}
\langle [x,y],z\rangle =\langle x,[y,z]\rangle.
\end{equation*}

Let $L\subseteq \mathrm{Lie}(S_A)$ be the orthogonal complement to $\mathrm{ker}(\phi)$. We show that $L$ is an ideal and $L\oplus \mathrm{ker}(\phi)=\mathrm{Lie}(S_A)$ as Lie algebras (the bracket between the two ideals vanishes). Suppose that $l\in L$ and $s\in\mathrm{Lie}(S_A)$. We see that for all $x\in\mathrm{ker}(\phi)$, we have $[x,s]\in\mathrm{ker}(\phi)$, and thus
\begin{equation*}
\langle x,[s,l]\rangle=\langle [x,s],l\rangle=0.
\end{equation*}
Hence, $[s,l]\in L$, so $L$ is an ideal. Finally, we see that if $x\in\mathrm{ker}(\phi)$ and $l\in L$, then as both $\mathrm{ker}(\phi)$ and $L$ are ideals, $[l,x]\in L\cap \mathrm{ker}(\phi)$. Hence, $||[l,x]||^2=0$, as it is orthogonal to itself. Hence, $[l,x]=0$, so $\mathrm{Lie}(S_A)=L\oplus\mathrm{ker}(\phi)$. 

By the isomorphism theorems, we know that as $\mathrm{im}(\pi_1|_{S_A})=\mathcal{S}$ is closed,
\begin{equation*}
L\simeq \mathrm{Lie}(S_A)/\mathrm{ker}(\phi)\simeq \mathrm{im}(\phi)=\mathrm{Lie}(\mathrm{im}(\pi_1|_{S_A}))=\mathrm{Lie}(\mathcal{S}).
\end{equation*}
Let $\psi\colon\mathrm{Lie}(\mathcal{S})\rightarrow L$ be the isomorphism.

Let $Y\subseteq S_A$ be the unique connected Lie subgroup corresponding to the Lie algebra $L\subseteq \mathrm{Lie}(S_A)$, whose existence is guaranteed by the Subgroups-subalgebras Theorem~\cite[Theorem~19.26]{lee_smooth_2012}. As $\tilde{\mathcal{S}}$ is simply-connected, the Homomorphisms Theorem tells us that there is a unique Lie group homomorphism $\Psi\colon\tilde{\mathcal{S}}\rightarrow Y$ such that $d_{e_{\tilde{\mathcal{S}}}}\Psi=\psi$~\cite[Theorem~20.19]{lee_smooth_2012}. Consider the map $\pi_1|_{S_A}\circ\Psi\colon\tilde{\mathcal{S}}\rightarrow \mathcal{S}$. We know that $d_{e_{\tilde{\mathcal{S}}}}(\pi_1|_{S_A}\circ\Psi)=\phi\circ\psi$. We show that this map is an isomorphism and we use it to find the smooth map we desire.

Indeed, if $\phi(\psi(x))=0$, then $\psi(x)\in L\cap\mathrm{ker}(\phi)=\{0\}$, so $x=0$, as $\psi$ is an isomorphism. Furthermore, consider $y\in\mathrm{Lie}(\mathcal{S})=\mathrm{im}(\phi)$. Then there is some $x\in\mathrm{Lie}(S_A)$ such that $\phi(x)=y$. We can uniquely write $x=l+z$ for $l\in L$ and $z\in\mathrm{ker}(\phi)$. But then $\phi(l)=y$. As $\psi$ is an isomorphism, there exists $w\in\mathrm{Lie}(\mathcal{S})$ such that $\psi(w)=l$. Therefore, $\phi(\psi(w))=y$, so $\phi\circ\psi$ is a Lie algebra isomorphism. Call the inverse of this map $g$.

By the Homomorphisms Theorem, we know that there is a unique Lie group homomorphism $G\colon \tilde{\mathcal{S}}\rightarrow\tilde{\mathcal{S}}$ such that $d_{e_{\tilde{\mathcal{S}}}}G=g$~\cite[Theorem~20.19]{lee_smooth_2012}. We show that $\Psi\circ G$ is the map we are searching for. Indeed, we have
\begin{equation*}
d_{e_{\tilde{\mathcal{S}}}}(\pi_1|_{S_A}\circ \Psi\circ G)=\phi\circ\psi\circ g=\mathrm{id}_{\mathrm{Lie}(\mathcal{S})}.
\end{equation*}
But $\pi\colon \tilde{\mathcal{S}}\rightarrow \mathcal{S}$ is a Lie group homomorphism whose pushforward at the identity is the identity. By the Homomorphisms Theorem, the two maps must be equal, so $\pi_1|_{S_A}\circ\Psi\circ G=\pi$~\cite[Theorem~20.19]{lee_smooth_2012}.

As $\psi\circ g$ is a Lie algebra homomorphism and $\phi\circ\psi\circ g=\mathrm{id}_{\mathrm{Lie}(\mathcal{S})}$, we have that $\psi\circ g(x)=(x,\rho(x))$ for some Lie algebra homomorphism $\rho\colon\mathrm{Lie}(\mathcal{S})\rightarrow \mathrm{Lie}(\mathcal{G})$. 

For all $x\in\mathrm{Lie}(\mathcal{S})$, we have $\Psi(G(\mathrm{exp}(x)))\in Y\subseteq S_A$. Noting that $\pi_1|_{S_A}\circ\Psi\circ G=\pi$, we see that for all $x\in\mathrm{Lie}(\mathcal{S})$ and $\theta\in\mathbb{R}$, 
\begin{equation*}
\mathrm{exp}(x\theta)\cdot \mathrm{exp}(\rho(x)\theta)\cdot A=A.
\end{equation*} 
Differentiating with respect to $\theta$ and evaluating at $\theta=0$, we obtain
\begin{equation*}
x.A+\rho(x).A=0.
\end{equation*}

Conversely, suppose that for some Lie algebra homomorphism $\rho$, we have for all $x\in\mathrm{Lie}(\mathcal{S})$, $x.A+\rho(x).A=0$. Consider $s\in\mathcal{S}$. As $\mathcal{S}$ is compact and connected, the exponential map is surjective. Hence, there is some $x\in\mathrm{Lie}(\mathcal{S})$ such that $\mathrm{exp}(x)=s$. We show that
\begin{equation*}
F(\theta):=\mathrm{exp}(x\theta)\cdot \mathrm{exp}(\rho(x)\theta)\cdot A,
\end{equation*}
is constant. 

First, recall that the symmetry action is a smooth map $\alpha\colon\mathcal{S}\times\mathcal{X}\rightarrow\mathcal{X}$. We also use the notation $s\cdot A$ to denote $\alpha(s,A)$. Given $s\in\mathcal{S}$, define $\iota_s\colon \mathcal{X}\rightarrow\mathcal{S}\times\mathcal{X}$ by $\iota_s(A):=(s,A)$. As the symmetry action is a left Lie group action, for all $s,s'\in\mathcal{S}$ and $A\in\mathcal{X}$, we have $\alpha(e_\mathcal{S},A)=A$ and 
\begin{equation}
\alpha\circ\iota_s\circ \alpha\circ\iota_{s'}(A)=\alpha(s,\alpha(s',A))=\alpha(ss',A)=\alpha\circ \iota_{ss'}(A).\label{eq:groupaction}
\end{equation}
Note that for $s\in\mathcal{S}$, $\alpha\circ\iota_s\colon\mathcal{X}\rightarrow\mathcal{X}$. For all $A\in\mathcal{X}$, $\alpha\circ\iota_s(A)=\alpha(s,A)=s\cdot A$, so we have $d_A(\alpha\circ\iota_s)\colon T_A\mathcal{X}\rightarrow T_{s.A}\mathcal{X}$. Define the smooth map $\tilde{\alpha}\colon \mathcal{S}\times T\mathcal{X}\rightarrow T\mathcal{X}$ as follows. For $s\in\mathcal{S}$ and $v\in T_A\mathcal{X}$, $\tilde{\alpha}(s,v):=d_A(\alpha\circ\iota_s)(v)\in T_{s.A}\mathcal{X}$. 

As $\alpha$ is a Lie group action of $\mathcal{S}$, so too is $\tilde{\alpha}$. We denote this group action by $s\cdot v:=\tilde{\alpha}(s,v)$. Note that $\alpha\circ\iota_{e_\mathcal{S}}=\mathrm{id}_{\mathcal{X}}$. Then we see that for all $s,s'\in\mathcal{S}$ and $v\in T_A\mathcal{X}$, by \eqref{eq:groupaction}, we have 
\begin{align*}
\tilde{\alpha}(e_\mathcal{S},v)&=d_A(\alpha\circ\iota_{e_\mathcal{S}})(v)=d_A\mathrm{id}_{\mathcal{X}}(v)=v\\
\tilde{\alpha}(s,\tilde{\alpha}(s',v))&=d_{s'.A}(\alpha\circ\iota_s)\circ d_A(\alpha\circ\iota_{s'})(v)\\
&=d_A(\alpha\circ\iota_s\circ \alpha\circ\iota_{s'})(v)\\
&=d_A(\alpha\circ\iota_{ss'})(v)\\
&=\tilde{\alpha}(ss',v).
\end{align*}

We can do the same thing for the gauge group $\mathcal{G}$. Returning to $F(\theta)$, we note that for $x\in\mathrm{Lie}(\mathcal{S})$, we have for any $\theta_0\in\mathbb{R}$, $\mathrm{exp}(\theta x)=\mathrm{exp}(\theta_0x)\cdot\mathrm{exp}((\theta-\theta_0)x)$. Thus, as the actions commute, we have
\begin{align*}
F'(\theta_0)&=\frac{d}{d\theta}\Bigr|_{\theta=\theta_0}\mathrm{exp}(\theta x)\cdot \mathrm{exp}(\theta \rho(x))\cdot A\\
&=\frac{d}{d\theta}\Bigr|_{\theta=\theta_0}\mathrm{exp}(\theta_0 x)\cdot \mathrm{exp}(\theta_0\rho(x))\cdot \mathrm{exp}((\theta-\theta_0)x)\cdot \mathrm{exp}((\theta-\theta_0)\rho(x))\cdot A.
\end{align*}
Letting $t:=\theta-\theta_0$, we have from the induced actions defined above
\begin{align*}
F'(\theta_0)&=\mathrm{exp}(\theta_0 x)\cdot \mathrm{exp}(\theta_0\rho(x))\cdot \left(\frac{d}{dt}\Bigr|_{t=0}\mathrm{exp}(tx)\cdot \mathrm{exp}(t\rho(x))\cdot A\right)\\
&=\mathrm{exp}(\theta_0 x)\cdot \mathrm{exp}(\theta_0\rho(x))\cdot (x.A+\rho(x).A)\\
&=0.
\end{align*}
Therefore, $\mathrm{exp}(x)\cdot \mathrm{exp}(\rho(x))\cdot A=F(1)=F(0)=A$. As $s=\mathrm{exp}(x)$ was arbitrary, $[A]$ is fixed by $\mathcal{S}$.
\end{proof}

\begin{note}
In Theorem~\ref{thm:mainthm}, we assume that the two Lie groups act on the left. Nothing changes in \eqref{eq:maineq} if instead we have two right group actions. However, if we have a mix of left and right group actions, then \eqref{eq:maineq} becomes $x.A-\rho(x).A=0$. This is because, in order to maintain its group structure, the definition of $S_A$ must change depending on which side the Lie groups act on. The same is true for Proposition~\ref{prop:R} and \eqref{eq:R}.\label{note:sign}
\end{note}

As $\mathcal{G}$ is compact in Theorem~\ref{thm:mainthm}, it can be viewed as a matrix Lie group. Hence, the Lie algebra homomorphism gives us a Lie algebra representation. Therefore, the problem of finding fixed points in the moduli space is one of representation theory. 

In some cases, we can get even more information about fixed points using representation theory. Suppose $\mathcal{X}$ is a vector space and $\mathcal{G}$ as well as $\mathcal{S}$ act linearly on $\mathcal{X}$, so that our actions are representations. For a symmetric point $A\in \mathcal{X}$, Theorem~\ref{thm:mainthm} tells us that $\mathrm{Lie}(\mathcal{S})$ acts on $A$ in two ways: through the action of $\mathrm{Lie}(\mathcal{S})$ directly and through the action of $\mathrm{Lie}(\mathcal{G})$ via $\rho$. Hence, there are some representations $(V_1,\lambda_1)$ and $(V_2,\lambda_2)$ of $\mathrm{Lie}(\mathcal{S})$ such that $A$ belongs to a trivial component of $(V_1,\lambda_1)\otimes (V_2,\lambda_2)$. This not only helps us determine the dimension of the vector space of solutions (the number of trivial components), but also tells us what representations $\rho$ to consider, as the tensor product must contain trivial components in order to have non-zero symmetric objects and the tensor product depends on $\rho$.

Theorem~\ref{thm:mainthm} requires compact groups. However, hope is not necessarily lost when you do not have compactness. Indeed, in previous work studying instantons with continuous symmetries, we do not have a compact gauge group. However, we are able to apply the same methods as in the proof of Theorem~\ref{thm:mainthm}. This is because on the subset of the space $\mathcal{X}$ corresponding to instantons, we are able to reduce the gauge group to a compact subgroup~\cite{lang_instantons_2024}.

In Proposition~\ref{prop:mainprop}, we see that when $\mathcal{S}$ is not assumed to be connected, we have two things to check: if $[A]$ is fixed by $\mathcal{S}_0$ and if it is fixed by $\mathcal{S}/\mathcal{S}_0$. 

Theorem~\ref{thm:mainthm} tells us that if $[A]$ is fixed by a compact and connected Lie group, then the gauge component of $S_A$ arises from a Lie algebra representation. We know that as $\mathcal{S}$ is compact, $\mathcal{S}/\mathcal{S}_0$ is a finite group. If $[A]$ is fixed by a finite group, it is not known if the gauge component of $S_A$ arises from a finite group representation.

Using a similar method, we prove Proposition~\ref{prop:R}.
%When $\mathcal{S}$ is one-dimensional, using a similar method, we obtain the same result even when we no longer have compactness. As an example of the use of this result, I use it to study instantons with circular symmetry in my thesis~\cite[Theorem~3.2.1]{lang_thesis_2024}.
%\begin{prop}
%Let $\mathcal{X}$ be a smooth manifold, $\mathcal{G}$ a Lie group, and $\mathcal{S}$ a connected, one-dimensional Lie group (isomorphic to either $S^1$ or $\mathbb{R}$). Suppose that $\mathcal{G}$ and $\mathcal{S}$ act smoothly on $\mathcal{X}$ on the left and the two actions commute. We have that $[A]\in\mathcal{X}/\mathcal{G}$ is fixed by $\mathcal{S}$ if and only if there is some $\rho\in\mathrm{Lie}(\mathcal{G})$ such that, for all $t\in\mathbb{R}$, \label{prop:R}
%\begin{equation}
%t.A+t\rho.A=0.\label{eq:R}
%\end{equation}
%\end{prop}

%\begin{note}
%The sign in \eqref{eq:R} changes just as mentioned in Note~\ref{note:sign}.
%\end{note}
%
\begin{proof}[Proof of Proposition~\ref{prop:R}]
We proceed as in the proof of Theorem~\ref{thm:mainthm}. For $[A]\in\mathcal{X}/\mathcal{G}$, fixed by $\mathcal{S}$, we define $S_A$, which is still a closed Lie subgroup of $\mathcal{S}\times\mathcal{G}$. However, $S_A$ is not necessarily compact.

We still have $\pi_1|_{S_A}\colon S_A\rightarrow \mathcal{S}$ is surjective and we consider $\phi:=d_{(e_\mathcal{S},e_\mathcal{G})}(\pi_1|_{S_A})\colon \mathrm{Lie}(S_A)\rightarrow \mathrm{Lie}(\mathcal{S})$. However, this time, as $S_A$ is not necessarily compact, we do not automatically have a $\mathrm{Ad}(S_A)$-invariant inner product on $\mathrm{Lie}(S_A)$. Instead, suppose, for the sake of contradiction, that $\mathrm{ker}(\phi)=\mathrm{Lie}(S_A)$. By the isomorphism theorems, we know that as $\mathrm{im}\left(\pi_1|_{S_A}\right)=\mathcal{S}$ is closed,
\begin{equation*}
0\simeq \mathrm{Lie}(S_A)/\mathrm{ker}(\phi)\simeq\mathrm{im}(\phi)=\mathrm{Lie}(\mathrm{im}(\pi_1|_{S_A}))=\mathrm{Lie}(\mathcal{S})\simeq\mathbb{R}.
\end{equation*}
Contradiction! 

Hence, $\mathrm{ker}(\phi)\subsetneq\mathrm{Lie}(S_A)$. That is, there exists $\xi:=(\xi_0,\xi_1)\in\mathrm{Lie}(S_A)$ such that $\xi_0\neq 0$, so $\phi(\xi)\neq 0$. Thus, $\phi(\xi)$ spans $\mathrm{Lie}(\mathcal{S})$. Let $\psi\colon\mathbb{R}\rightarrow\langle\xi\rangle\subseteq\mathrm{Lie}(S_A)$ be the Lie algebra homomorphism defined by $\psi(\alpha):=\alpha\xi$.

Let $Y\subseteq S_A$ be the unique connected Lie subgroup corresponding to the Lie algebra $\langle\xi\rangle\subseteq\mathrm{Lie}(S_A)$, whose existence is guaranteed by the Subgroups-subalgebras Theorem~\cite[Theorem~19.26]{lee_smooth_2012}. 

The rest of the proof proceeds as in the proof of Theorem~\ref{thm:mainthm}. Note that $\phi\circ\psi$ is an isomorphism as it is bijective and automatically a Lie algebra homomorphism. Indeed, $\phi(\psi(1))=\phi(\xi)\neq 0$. For the converse direction, the only note is that even when $\mathcal{S}\simeq\mathbb{R}$, the exponential map is surjective.
\end{proof}

In the proof of Proposition~\ref{prop:R}, we see that we cannot use the same technique for other Lie groups, abelian or not. Indeed, using the same notation as in the proof, in general we have the Lie algebra isomorphism $\mathrm{Lie}(S_A)/\mathrm{ker}(\phi)\simeq\mathrm{Lie}(\mathcal{S})$. 

Thus, we can always find elements of $\mathrm{Lie}(S_A)$ whose images under $\phi$ span $\mathrm{Lie}(\mathcal{S})$. However, in general, we do not know that these elements span a Lie algebra in $S_A$. Indeed, we only know that they span a Lie algebra in the quotient $\mathrm{Lie}(S_A)/\mathrm{ker}(\phi)$. The case of $\mathcal{S}$ being one-dimensional is special, as any non-zero element of $\mathrm{Lie}(S_A)$ spans a one-dimensional Lie algebra.

\subsection{An example using Proposition~\ref{prop:R}}\label{subsec:example}
In this section, we construct an example and use Proposition~\ref{prop:R} to find symmetric points. Previous work has used Theorem~\ref{thm:mainthm} and Proposition~\ref{prop:R} to study symmetric gauge theoretic objects in more involved case studies~\cite{charbonneau_construction_2022, lang_hyperbolic_2023, lang_thesis_2024, lang_instantons_2024}.
\begin{example}
Let $\mathcal{X}:=\mathbb{C}^2$ and $\mathcal{G}:= S^1=:\mathcal{S}$. Suppose that we only care about pairs in $\mathbb{C}^2$ up to their modulus and relative phase. That is, suppose that $\mathcal{G}$ acts on $\mathcal{X}$ as ${^{e^{i\theta}} (z,w)}:=\left(e^{i\theta}z,e^{i\theta}w\right)$.

Suppose that $\mathcal{S}$ acts on $\mathcal{X}$ as ${_{e^{i\phi}}(z,w)}:=\left(\cos(\phi)z+\sin(\phi)w,-\sin(\phi)z+\cos(\phi)w\right)$. Note that the actions of $\mathcal{S}$ and $\mathcal{G}$ commute. The only points symmetric under $\mathcal{S}$ are $\{(z,\mp iz)\mid z\in\mathbb{C}\}$.
\end{example}

\begin{proof}
Proposition~\ref{prop:R} tells us that $[(z,w)]\in\mathcal{X}/\mathcal{G}$ is fixed by $\mathcal{S}$ if and only if there is some $\rho\in\mathrm{Lie}(\mathcal{G})=\mathbb{R}$ satisfying \eqref{eq:R}. 

The Lie algebra actions of $\mathcal{S}$ and $\mathcal{G}$ on $\mathcal{X}$ are, respectively, given by
\begin{align*}
{{_\phi}(z,w)}&:=\left.\frac{d}{dt}\right|_{t=0}{_{e^{it\phi}}(z,w)}=(\phi w,-\phi z) \quad\textrm{and} \\
{{^\theta}(z,w)}&:=\left.\frac{d}{dt}\right|_{t=0}{_{e^{it\theta}}(z,w)}=(i\theta z,i\theta w).
\end{align*}

Suppose that $[(z,w)]$ is fixed by $\mathcal{S}$. Then there is some $\rho\in\mathbb{R}$ such that ${{_\phi}(z,w)}+{{^{\rho\phi}}(z,w)}=(0,0)$. Substituting the Lie algebra actions, we have that for all $\phi\in\mathbb{R}$,
\begin{equation*}
(\phi w,-\phi z)+(i\rho\phi z,i\rho\phi w)=(0,0).
\end{equation*}
As this equality holds for all $\phi\in\mathbb{R}$, we have $w=-i\rho z$ and $-z+i\rho w=0$. Substituting the former into the latter, we obtain $z(1-\rho^2)=0$. Thus, $w=z=0$ or $\rho=\pm 1$. Moreover, if $\rho=\pm 1$, then $w=\mp iz$. 

As all symmetric points must satisfy the above constraints, we have that the only points symmetric under $\mathcal{S}$ are $\{(z,\mp iz)\mid z\in\mathbb{C}\}$.
\end{proof}

\subsection{Closed sets}\label{subsec:closed}
Armed with Theorem~\ref{thm:mainthm}, we can examine objects that are fixed by different symmetries. It turns out that we can ignore certain symmetries, as we can investigate their closure instead.
%\begin{prop}
%Let $\mathcal{X}$ be a smooth manifold, $\mathcal{G}$ a compact Lie group, and $\mathcal{S}$ a Lie group. Suppose that $\mathcal{G}$ and $\mathcal{S}$ act on $\mathcal{X}$ and their actions commute. 
%For $A\in\mathcal{X}$, let $H_A:=\{s\in\mathcal{S}\mid s\cdot [A]=[A]\}$. Then $H_A$ is a closed Lie subgroup of $\mathcal{S}$. Moreover, if $\mathcal{S}$ is compact, then $H_A$ is as well.
%\label{prop:closed}
%\end{prop}
%
%\begin{proof}
%For $A\in\mathcal{X}$, let $S_A:=\{(s,g)\mid s\cdot g\cdot A=A\}\subseteq\mathcal{S}\times\mathcal{G}$. From the proof of Theorem~\ref{thm:mainthm}, we know that $S_A$ is closed (only the group structure of $S_A$ relies on the sides on which the gauge and symmetry groups act).
%
%As $\mathcal{G}$ is compact, we have that $\pi_1\colon\mathcal{S}\times\mathcal{G}\rightarrow\mathcal{S}$ is a closed map. Thus, $\pi_1(S_A)$ is closed. Finally, we have $s\in \pi_1(S_A)$ if and only if there is some $g\in\mathcal{G}$ such that $s\cdot g\cdot A=A$. That is, $s\cdot [A]=[A]$. Hence, $H_A=\pi_1(S_A)$ is closed. Moreover, $H_A$ is a subgroup of $\mathcal{S}$. As a closed subgroup of a Lie group, by the Closed-subgroup Theorem $H_A$ itself is a closed Lie subgroup of $\mathcal{S}$~\cite[Theorem~20.12]{lee_smooth_2012}. Moreover, if $\mathcal{S}$ is compact, then $H_A$ is as well.
%\end{proof}
\begin{prop}
Let $\mathcal{X}$ be a smooth manifold and let both $\mathcal{G}$ and $\mathcal{S}$ be Lie groups. Suppose that $\mathcal{G}$ and $\mathcal{S}$ act on $\mathcal{X}$ and their actions commute. Additionally, suppose that $\mathcal{G}$ acts on $\mathcal{X}$ properly. For $A\in\mathcal{X}$, let $H_A:=\{s\in \mathcal{S}\mid s\cdot [A]=[A]\}$. Then $H_A$ is a closed Lie subgroup of $\mathcal{S}$. Moreover, if $\mathcal{S}$ is compact, then $H_A$ is as well.\label{prop:closed}
\end{prop}

\begin{proof}
% Let $f\colon\mathcal{S}\times\mathcal{X}\rightarrow\mathcal{X}$ be the smooth action of $\mathcal{S}$ on $\mathcal{X}$. Let $\overline{f}\colon\mathcal{S}\times\mathcal{X}/\mathcal{G}\rightarrow\mathcal{X}/\mathcal{G}$ be the induced action $\overline{f}(s,[A]):=[f(s,A)]$. We show that $\overline{f}$ is continuous with respect to the product topology.
%
%Let $U\subseteq \mathcal{X}/\mathcal{G}$ be open. Then $\pi^{-1}(U)\subseteq\mathcal{X}$ is open. As $f$ is continuous, $f^{-1}(\pi^{-1}(U))\subseteq\mathcal{S}\times\mathcal{X}$ is open. Let $\Pi\colon\mathcal{S}\times\mathcal{X}\rightarrow\mathcal{S}\times\mathcal{X}/\mathcal{G}$ be the continuous map defined by $\Pi(s,A):=(s,[A])$. As $\pi$ is an open map, we have that $\Pi$ is an open map. We know that $(s,A)\in\Pi^{-1}(\overline{f}^{-1}(U))$ if and only if $[f(s,A)]\in U$, which happens if and only if $(s,A)\in f^{-1}(\pi^{-1}(U))$. Thus, $\Pi^{-1}(\overline{f}^{-1}(U))=f^{-1}(\pi^{-1}(U))$ is open. As $\Pi$ is an open map, $\overline{f}^{_1}(U)$ is open. That is, $\overline{f}$ is continuous in the product topology.
Suppose that $s$ is a limit point of $H_A$. Then there is a sequence $\{s_i\}_{i=1}^\infty\subseteq H_A$ such that $s_i\rightarrow s$. Consider $\{s_i\cdot [A]\}_{i=1}^\infty\subseteq \mathcal{X}/\mathcal{G}$. Note that, for all $i\in\mathbb{N}$, we have $s_i\cdot [A]=[A]$, as $s_i\in H_A$. As $\mathcal{G}$ acts properly, $\mathcal{X}/\mathcal{G}$ is Hausdorff, so limits are unique, meaning $s_i\cdot [A]\rightarrow [A]$~\cite[Proposition~21.4]{lee_smooth_2012}.

Due to the continuity of the action of $\mathcal{G}$ on $\mathcal{X}$, the quotient map $\pi\colon\mathcal{X}\rightarrow\mathcal{X}/\mathcal{G}$ is an open map~\cite[Lemma~21.1]{lee_smooth_2012}. It then follows that the smooth action of $\mathcal{S}$ descends to a continuous action on the moduli space $\mathcal{X}/\mathcal{G}$. Because of this continuity, we see that
\begin{equation*}
[A]=\lim_{i\rightarrow\infty} s_i\cdot [A]=\left(\lim_{i\rightarrow\infty} s_i\right)\cdot [A]=s\cdot [A].
\end{equation*}
Therefore, $s\in H_A$, so $H_A$ contains all its limit points. Hence, $H_A$ is closed.

Moreover, $H_A$ is a subgroup of $\mathcal{S}$. As a closed subgroup of a Lie group, by the Closed-subgroup Theorem $H_A$ itself is a closed Lie subgroup of $\mathcal{S}$~\cite[Theorem~20.12]{lee_smooth_2012}. Finally, if $\mathcal{S}$ is compact, then $H_A$ is as well.
\end{proof}

In particular, note that if $\mathcal{G}$ is compact, then it automatically acts properly. In this case, we can prove the above result another way. Indeed, recall the definition of $S_A$ from the proof of Theorem~\ref{thm:mainthm}. Even in this more general setting, $S_A$ is still closed. Projecting onto the first coordinate, we have $H_A=\pi_1(S_A)$. As $\mathcal{G}$ is compact, $\pi_1$ is a closed map, so $H_A$ is closed.

Proposition~\ref{prop:closed} requires a proper action. However, proving that an action is proper is challenging. It is much simpler when the group is compact. However, just as before, hope is not necessarily lost when you do not have compactness. Indeed, in previous work studying instantons with continuous symmetries, we do not have a compact gauge group. However, we are able to apply the same methods as above, when $\mathcal{G}$ is compact~\cite[Proposition~3.4]{lang_instantons_2024}.

As an example of the importance of Proposition~\ref{prop:closed}, suppose we look at objects symmetric under different connected Lie subgroups of $S^1\times S^1$. Excluding the trivial subgroup, we must consider the Lie subgroups $R_t:=\{(e^{i\theta},e^{it\theta})\mid \theta\in\mathbb{R}\}$, for all $t\in\mathbb{R}$, $R_\infty:=\{1\}\times S^1$, and $S^1\times S^1$. 

When $t\in\mathbb{Q}$, $R_t\simeq S^1$. Otherwise, $R_t\simeq \mathbb{R}$, which is a problem as $\mathbb{R}$ is not compact. Proposition~\ref{prop:R} allows us to study objects symmetric under $R_t$, for any $t$, as $\mathbb{R}$ is one-dimensional. However, Proposition~\ref{prop:closed} tells us that, in our case, we can ignore those symmetries that are not compact. Indeed, when $t\notin\mathbb{Q}$, $R_t$ is dense in $S^1\times S^1$, so an element in the moduli space is fixed by $R_t$ if and only if it is fixed by $S^1\times S^1$. Thus, we are only left with compact and connected Lie groups for symmetry groups, exactly those we can study using Theorem~\ref{thm:mainthm}.

\section*{Acknowledgements}

Many thanks to Benoit Charbonneau for many useful discussions along the path to Theorem~\ref{thm:mainthm} as well as Ben Webster and Tom Baird for providing additional references of group actions on moduli spaces. 

\bibliographystyle{halpha}
\bibliography{./Files/Bibliography/bibliography.bib}

\newcommand{\etalchar}[1]{$^{#1}$}
\begin{thebibliography}{BCFGP19}
\expandafter\ifx\csname url\endcsname\relax
  \def\url#1{\texttt{#1}}\fi
\expandafter\ifx\csname doi\endcsname\relax
  \def\doi#1{\burlalt{doi:#1}{http://dx.doi.org/#1}}\fi
\expandafter\ifx\csname urlprefix\endcsname\relax\def\urlprefix{URL: }\fi
\expandafter\ifx\csname href\endcsname\relax
  \def\href#1#2{#2}\fi
\expandafter\ifx\csname burlalt\endcsname\relax
  \def\burlalt#1#2{\href{#2}{#1}}\fi

\bibitem[And97]{andersen_fixed_1997}
J.~Andersen.
\newblock Fixed points of the mapping class group in the {SU(n)} moduli spaces.
\newblock {\em Proceedings of the American Mathematical Society},
  125(5):1511--1515, 1997.
\newblock \doi{10.1090/S0002-9939-97-03788-X}.

\bibitem[AS13]{allen_adhm_2013}
J.~P. Allen and P.~M. Sutcliffe.
\newblock {ADHM} polytopes.
\newblock {\em Journal of High Energy Physics}, 2013(5):63, 2013,
  \burlalt{arXiv:1302.4664}{http://arxiv.org/abs/arXiv:1302.4664}.
\newblock \doi{10.1007/JHEP05(2013)063}.

\bibitem[BA90]{braam_boundary_1990}
P.~J. Braam and D.~M. Austin.
\newblock Boundary values of hyperbolic monopoles.
\newblock {\em Nonlinearity}, 3(3):809, 1990.
\newblock \doi{10.1088/0951-7715/3/3/012}.

\bibitem[Bai24]{baird_cohomology_2024}
T.~J. Baird.
\newblock Cohomology of fixed point sets of anti-symplectic involutions in the
  {Hilbert} scheme of points on a surface, 2024,
  \burlalt{arXiv:2311.16287}{http://arxiv.org/abs/arXiv:2311.16287}.

\bibitem[BCFGP19]{biswas_involutions_2019}
I.~Biswas, L.~A. Calvo, E.~Franco, and O.~Garc\'{i}a-Prada.
\newblock Involutions of {Higgs} moduli spaces over elliptic curves and
  pseudo-real {Higgs} bundles.
\newblock {\em Journal of Geometry and Physics}, 142:47--65, 2019,
  \burlalt{arXiv:1612.08364}{http://arxiv.org/abs/arXiv:1612.08364}.
\newblock \doi{10.1016/j.geomphys.2019.03.014}.

\bibitem[BE10]{braden_tetrahedrally_2010}
H.~W. Braden and V.~Z. Enolski.
\newblock On the {Tetrahedrally} {Symmetric} {Monopole}.
\newblock {\em Communications in Mathematical Physics}, 299(1):255--282, 2010,
  \burlalt{arXiv:0908.3449}{http://arxiv.org/abs/arXiv:0908.3449}.
\newblock \doi{10.1007/s00220-010-1081-0}.

\bibitem[Bec20]{beckett_equivariant_2020}
M.~Beckett.
\newblock {\em Equivariant {Nahm} {Transforms} and {Minimal} {Yang}--{Mills}
  {Connections}}.
\newblock {PhD} {Thesis}, Duke University, Durham, NC, 2020.
\newblock \urlprefix\url{https://hdl.handle.net/10161/20982}.

\bibitem[BGP15]{biswas_anti-holomorphic_2015}
I.~Biswas and O.~Garc\'{i}a-Prada.
\newblock Anti-holomorphic involutions of the moduli spaces of {Higgs} bundles.
\newblock {\em Journal de l’École polytechnique — Mathématiques},
  2:35--54, 2015,
  \burlalt{arXiv:1401.7236}{http://arxiv.org/abs/arXiv:1401.7236}.
\newblock \doi{10.5802/jep.16}.

\bibitem[CDL{\etalchar{+}}22]{charbonneau_construction_2022}
B.~Charbonneau, A.~Dayaprema, C.~J. Lang, Á. Nagy, and H.~Yu.
\newblock Construction of {Nahm} data and {BPS} monopoles with continuous
  symmetries.
\newblock {\em Journal of Mathematical Physics}, 63(1):013507, 2022,
  \burlalt{arXiv:2102.01657}{http://arxiv.org/abs/arXiv:2102.01657}.
\newblock \doi{10.1063/5.0055913}.

\bibitem[Coc14]{cockburn_symmetric_2014}
A.~Cockburn.
\newblock Symmetric hyperbolic monopoles.
\newblock {\em Journal of Physics A: Mathematical and Theoretical},
  47(39):395401, 2014.
\newblock \doi{10.1088/1751-8113/47/39/395401}.

\bibitem[Cor18]{cork_symmetric_2018}
J.~Cork.
\newblock Symmetric calorons and the rotation map.
\newblock {\em Journal of Mathematical Physics}, 59(6):062902, 2018,
  \burlalt{arXiv:1711.04599}{http://arxiv.org/abs/arXiv:1711.04599}.
\newblock \doi{10.1063/1.5017193}.

\bibitem[GPB20]{garcia-prada_finite_2020}
O.~Garc\'{i}a-Prada and S.~Basu.
\newblock Finite group actions on {Higgs} bundle moduli spaces and twisted
  equivariant structures, 2020,
  \burlalt{arXiv:2011.04017}{http://arxiv.org/abs/arXiv:2011.04017}.

\bibitem[GPR19]{garcia-prada_involutions_2019}
O.~Garc\'{i}a-Prada and S.~Ramanan.
\newblock Involutions and higher order automorphisms of {Higgs} bundle moduli
  spaces.
\newblock {\em Proceedings of the London Mathematical Society},
  119(3):681--732, 2019,
  \burlalt{arXiv:1605.05143}{http://arxiv.org/abs/arXiv:1605.05143}.
\newblock \doi{10.1112/plms.12242}.

\bibitem[GPW20]{garcia-prada_action_2020}
O.~Garc\'{i}a-Prada and G.~Wilkin.
\newblock Action of the {Mapping} {Class} {Group} on {Character} {Varieties}
  and {Higgs} {Bundles}.
\newblock {\em Documenta Mathematica}, 25:841--868, 2020,
  \burlalt{arXiv:1612.02508}{http://arxiv.org/abs/arXiv:1612.02508}.
\newblock \doi{10.4171/dm/764}.

\bibitem[HMM95]{hitchin_symmetric_1995}
N.~J. Hitchin, N.~S. Manton, and M.~K. Murray.
\newblock Symmetric monopoles.
\newblock {\em Nonlinearity}, 8(5):661--692, 1995,
  \burlalt{arXiv:hep-th/9407102}{http://arxiv.org/abs/arXiv:hep-th/9407102}.
\newblock \doi{10.1088/0951-7715/8/5/002}.

\bibitem[Hou99]{houghton_instanton_1999}
C.~J. Houghton.
\newblock Instanton vibrations of the 3-{Skyrmion}.
\newblock {\em Physical Review D}, 60(10):105003, 1999,
  \burlalt{arXiv:hep-th/9905009}{http://arxiv.org/abs/arXiv:hep-th/9905009}.
\newblock \doi{10.1103/PhysRevD.60.105003}.

\bibitem[HS96a]{houghton_octahedral_1996}
C.~J. Houghton and P.~M. Sutcliffe.
\newblock Octahedral and dodecahedral monopoles.
\newblock {\em Nonlinearity}, 9(2):385--401, 1996,
  \burlalt{arXiv:hep-th/9601147}{http://arxiv.org/abs/arXiv:hep-th/9601147}.
\newblock \doi{10.1088/0951-7715/9/2/005}.

\bibitem[HS96b]{houghton_tetrahedral_1996}
C.~J. Houghton and P.~M. Sutcliffe.
\newblock Tetrahedral and cubic monopoles.
\newblock {\em Communications in Mathematical Physics}, 180(2):343--361, 1996,
  \burlalt{arXiv:hep-th/9601146}{http://arxiv.org/abs/arXiv:hep-th/9601146}.
\newblock \doi{10.1007/BF02099717}.

\bibitem[HS97]{houghton_sun_1997}
C.~J. Houghton and P.~M. Sutcliffe.
\newblock {SU}({N}) monopoles and {Platonic} symmetry.
\newblock {\em Journal of Mathematical Physics}, 38(11):5576--5589, 1997,
  \burlalt{arXiv:hep-th/9708006}{http://arxiv.org/abs/arXiv:hep-th/9708006}.
\newblock \doi{10.1063/1.532152}.

\bibitem[HS19]{hoskins_group_2019}
V.~Hoskins and F.~Schaffhauser.
\newblock Group actions on quiver varieties and applications.
\newblock {\em International Journal of Mathematics}, 30(02):1950007, 2019.
\newblock \doi{10.1142/S0129167X19500071}.

\bibitem[HS20]{hoskins_rational_2020}
V.~Hoskins and F.~Schaffhauser.
\newblock Rational points of quiver moduli spaces.
\newblock {\em Annales de l'Institut Fourier}, 70(3):1259--1305, 2020,
  \burlalt{arXiv:1704.08624}{http://arxiv.org/abs/arXiv:1704.08624}.
\newblock \doi{10.5802/aif.3334}.

\bibitem[Lan24a]{lang_hyperbolic_2023}
C.~J. Lang.
\newblock Hyperbolic monopoles with continuous symmetries.
\newblock {\em Journal of Geometry and Physics}, 203:105258, 2024,
  \burlalt{arXiv:2310.10626}{http://arxiv.org/abs/arXiv:2310.10626}.
\newblock \doi{10.1016/j.geomphys.2024.105258}.

\bibitem[Lan24b]{lang_thesis_2024}
C.~J. Lang.
\newblock {\em Solitons with continuous symmetries}.
\newblock {Ph.D.} {Thesis}, University of Waterloo, Waterloo, Canada, 2024.
\newblock \urlprefix\url{https://hdl.handle.net/10012/20906}.

\bibitem[Lan25]{lang_instantons_2024}
C.~J. Lang.
\newblock Instantons with continuous conformal symmetries: Hyperbolic and
  singular monopoles and more, oh my!
\newblock 2025,
  \burlalt{arXiv:2501.07406}{http://arxiv.org/abs/arXiv:2501.07406}.

\bibitem[Lee12]{lee_smooth_2012}
J.~M. Lee.
\newblock {\em Introduction to {Smooth} {Manifolds}}.
\newblock {Graduate} {Texts} in {Mathematics}. Springer New York, New York, NY,
  2nd edition, 2012.
\newblock \doi{10.1007/978-1-4419-9982-5}.

\bibitem[LM94]{leese_stable_1994}
R.~A. Leese and N.~S. Manton.
\newblock Stable instanton-generated {Skyrme} fields with baryon numbers three
  and four.
\newblock {\em Nuclear Physics A}, 572(3-4):575--599, 1994.
\newblock \doi{10.1016/0375-9474(94)90401-4}.

\bibitem[MS14]{manton_platonic_2014}
N.~S. Manton and P.~M. Sutcliffe.
\newblock Platonic hyperbolic monopoles.
\newblock {\em Communications in Mathematical Physics}, 325(3):821--845, 2014,
  \burlalt{arXiv:1207.2636}{http://arxiv.org/abs/arXiv:1207.2636}.
\newblock \doi{10.1007/s00220-013-1864-1}.

\bibitem[Nak99]{nakajima_lectures_1999}
H.~Nakajima.
\newblock {\em Lectures on {Hilbert} {Schemes} of {Points} on {Surfaces}},
  volume~18 of {\em University {Lecture} {Series}}.
\newblock American Mathematical Society, Providence, RI, 1999.
\newblock \doi{10.1090/ulect/018}.

\bibitem[Nak01]{nakajima_quiver_2001}
H.~Nakajima.
\newblock Quiver varieties and finite dimensional representations of quantum
  affine algebras.
\newblock {\em Journal of the American Mathematical Society}, 14(1):145--238,
  2001, \burlalt{arXiv:math/9912158}{http://arxiv.org/abs/arXiv:math/9912158}.
\newblock \doi{10.1090/S0894-0347-00-00353-2}.

\bibitem[NR75]{narasimhan_generalised_1975}
M.~S. Narasimhan and S.~Ramanan.
\newblock Generalised {Prym} varieties as fixed points.
\newblock {\em Journal of the Indian Mathematical Society}, 39(1–4):1--19,
  1975.

\bibitem[Sch16]{schaffhauser_finite_2016}
F.~Schaffhauser.
\newblock Finite group actions on moduli spaces of vector bundles.
\newblock {\em Séminaire de théorie spectrale et géométrie}, 34:33--63,
  2016.
\newblock \doi{10.5802/tsg.354}.

\bibitem[SS99]{singer_symmetric_1999}
M.~A. Singer and P.~M. Sutcliffe.
\newblock Symmetric instantons and {Skyrme} fields.
\newblock {\em Nonlinearity}, 12(4):987, 1999,
  \burlalt{arXiv:hep-th/9901075}{http://arxiv.org/abs/arXiv:hep-th/9901075}.
\newblock \doi{10.1088/0951-7715/12/4/315}.

\bibitem[Sut04]{sutcliffe_instantons_2004}
P.~M. Sutcliffe.
\newblock Instantons and the {Buckyball}.
\newblock {\em Proceedings of the Royal Society of London. Series A:
  Mathematical, Physical and Engineering Sciences}, 460(2050):2903--2912, 2004,
  \burlalt{arXiv:hep-th/0309157}{http://arxiv.org/abs/arXiv:hep-th/0309157}.
\newblock \doi{10.1098/rspa.2004.1325}.

\bibitem[Sut05]{sutcliffe_platonic_2005}
P.~M. Sutcliffe.
\newblock Platonic instantons.
\newblock {\em Czechoslovak Journal of Physics}, 55(11):1515--1520, 2005.
\newblock \doi{10.1007/s10582-006-0034-5}.

\bibitem[Whi22]{whitehead_integrality_2022}
S.~Whitehead.
\newblock Integrality theorems for symmetric instantons.
\newblock {M.Math} {Thesis}, University of Waterloo, Waterloo, Ontario, Canada,
  2022.
\newblock \urlprefix\url{http://hdl.handle.net/10012/18487}.

\end{thebibliography}
\end{document}